\documentclass[]{amsart}

\author{Matthew D. Kvalheim}

\address{Department of Mathematics and Statistics, University of Maryland, Baltimore County, Baltimore, MD, USA}

\email{kvalheim@umbc.edu}

\makeatletter
\@namedef{subjclassname@2020}{%
	\textup{2020} Mathematics Subject Classification}
\makeatother   % MK: overriding since I only have up to 2010 classifications in my version of amsart
\title%[Stabilizability of first-order within second-order systems]
[Stabilizability of first-order dynamics in second-order systems]{Stabilizability of first-order dynamics in second-order systems}
\usepackage{import} % Use import instead of input; for pdf_tex in figure folder

\usepackage{cite}
\usepackage{amsmath}
\usepackage{amssymb}
\usepackage{stmaryrd} % for minnorm \llfloor, \rrfloor symbols
\usepackage{mathtools}
\usepackage{tikz-cd}
\usepackage{graphicx}
\usepackage{subfig}
\usepackage{tikz-cd}

\DeclareFontFamily{U}{mathx}{}
\DeclareFontShape{U}{mathx}{m}{n}{
	<5> <6> <7> <8> <9> <10>
	<10.95> <12> <14.4> <17.28> <20.74> <24.88>
	mathx10
}{}
\DeclareSymbolFont{mathx}{U}{mathx}{m}{n}
\DeclareMathAccent{\widecheck}{0}{mathx}{"71}

\newcommand{\concept}[1]{\textbf{#1}}

\newcommand{\N}{\mathbb{N}}

\newcommand{\R}{\mathbb{R}}

\newcommand{\slot}{\,\cdot\,} % Empty argument slot

\newcommand{\T}{T}
\newcommand{\Nor}{N}
\newcommand{\id}{\textnormal{id}}

\newcommand{\cfg}{Q}
\newcommand{\st}{M}
\newcommand{\att}{V}
\newcommand{\nbhd}{W}
\newcommand{\nbhdt}{\tilde{W}}
\newcommand{\rmo}{\kappa}
\newcommand{\spray}{S}
\newcommand{\rmt}{\nu}
\newcommand{\q}{q}
\newcommand{\cs}{F}
\newcommand{\ctrl}{U}
\newcommand{\cp}{\pi_{\ctrl}} % control space projection for use in section~\ref{subsec:underactuated-proof-prelims} only
\newcommand{\uc}{u}

\newcommand{\sect}{\Gamma}
\newcommand{\vf}{\mathfrak{X}}
\newcommand{\vo}{X}
\newcommand{\vt}{Y}

\newcommand{\vtt}{\tilde{Y}}
\newcommand{\adv}{\vt}
\newcommand{\dbn}{\mathcal{D}}
\newcommand{\vl}{\textnormal{vl}} % From p. 107 of Michor "topics..."
\newcommand{\V}{V}

\newcommand{\ip}[2]{\langle #1, #2 \rangle}
\DeclarePairedDelimiter\norm{\lVert}{\rVert}
\DeclarePairedDelimiter\minnorm{\llfloor}{\rrfloor}

\newcommand{\MK}[1]{[{\color{magenta}(M.D.K.) \color{cyan}\textsc{#1}}]}

\theoremstyle{definition}
\newtheorem{Def}{Definition}
\newtheorem{Lem}{Lemma}
\newtheorem{Th}{Theorem}

\newtheorem{Prop}{Proposition}
\newtheorem{Assump}{Assumption}
\newtheorem{Rem}{Remark}
\newtheorem{Ex}{Example}
\newtheorem*{AdvCond}{Adversary condition}

\newtheorem*{AdvThm}{Adversary theorem}

\usepackage{marvosym}
\usepackage[%hyperref
hypertexnames,%
citecolor=blue,%
colorlinks=true,%
linkcolor=red%
]{hyperref}
\usepackage[all]{hypcap}

\begin{document}

\begin{abstract}
We study whether second-order systems can be made to behave like prescribed first-order dynamical systems through feedback control.
More precisely, we study whether prescribed vector fields on compact smooth manifolds, viewed geometrically as sections of the tangent bundle, can be asymptotically stabilized in a strong sense by second-order control systems on the base manifold.
Our class of second-order systems includes most Lagrangian systems, and we obtain both positive and negative results.
The positive result asserts that, for fully actuated systems, the section corresponding to any smooth vector field can be made globally exponentially stable, normally hyperbolic, and more.
In particular, not only does each closed-loop solution asymptotically have the prescribed velocities, but it also converges to a trajectory of the first-order dynamics generated by the prescribed vector field at an exponential rate.
Thus, the closed-loop second-order system asymptotically reproduces the prescribed first-order dynamics.
In contrast, the negative result asserts that, for underactuated systems on manifolds with nonzero Euler characteristic, sections corresponding to ``almost all'' smooth vector fields cannot even be locally asymptotically stabilized.
This includes, in particular, all vector fields with only isolated zeros.
An example shows that the Euler characteristic assumption is necessary for the negative result.
\end{abstract}

\maketitle

\section{Introduction}
We study general second-order control systems
\begin{equation}\label{eq:riemannian-cs}
	\nabla_{\dot{\q}}\dot{\q} = \cs(\uc)
\end{equation}
on a smooth Riemannian manifold $\cfg$, where $\nabla$ is the Levi-Civita connection.
%See \cite{abraham1978foundations,bloch2015nonholonomic,bullo2005geometric}  for related standard details.
In the general framework here,  $\cs\colon \ctrl\to \T\cfg$ is a map making the diagram
\begin{equation}\label{eq:F-cd}
	\begin{tikzcd}
		\ctrl \arrow[r,"F"] \arrow[d,"\pi_{\ctrl}",swap] & \T\cfg \arrow[d,"\pi_{\T\cfg}"]\\
		 \T \cfg \arrow[r, "\pi_{\T\cfg}",swap] & \cfg
	\end{tikzcd}
\end{equation}
commute, where $\pi_{\T \cfg}$ is the tangent bundle projection and $\pi_{\ctrl}$  is the control space projection.
This framework extends a second-order version of Brockett's ``bundle picture of control'' \cite[p.~15]{brockett1977control}, \cite[sec.~4.6]{bloch2015nonholonomic}, and includes all regular Lagrangian control systems \cite[Def.~3.5.8]{abraham1978foundations} such as the standard class of ``manipulator equations'' in robotics with nondegenerate inertia tensor \cite[eq.~4.2.4]{murray1994mathematical}, written locally as
\begin{equation}\label{eq:manipulator-equations}
	M(\q)\ddot{\q} + C(\q,\dot{\q})\dot{\q} + G(\q) = \tau(\q,\dot{\q},\uc).
\end{equation}
\begin{Rem}\label{rem:no-redundant-expressions}
	Since we work globally in this paper, we do not write redundant expressions such as ``$\cs(\q,\dot{\q},\uc)$'' in \eqref{eq:riemannian-cs} or elsewhere, as $\uc\in \ctrl$  contains all information about $\dot{\q}=\pi_{\ctrl}(\uc)$, which in turn contains all information about $\q = \pi_{\T\cfg}(\dot{\q})$. 
\end{Rem}

This general framework is motivated by robotics \cite{koditschek2021robotics}, spacecraft \cite{leve2025spacecraft}, and other modern applications for which configuration spaces are non-Euclidean and/or the combined space of states and controls is not globally a product \cite{leve2025future}.
A number of recent works have employed similarly general frameworks for first-order systems \cite{kvalheim2022necessary,kvalheim2023obstructions, baryshnikov2023topological,kvalheim2025relationships,belabbas2025why,deSa2025bundles}.

For such general as well as classical second-order systems, a natural goal is to make the configuration $\q$ follow ``reference dynamics'' prescribed by a system of first-order ordinary differential equations
\begin{equation}\label{eq:first-order-autonomous-ode}
\dot{\q} = \vo(\q),
\end{equation}
where $\vo$ is a vector field on the configuration space $\cfg$.
Here ``reference dynamics'' should be contrasted with the distinct notion of  a ``reference trajectory'' $t\mapsto r(t)\in Q$ \cite{bullo1999tracking,lee2010geometric, welde2024almost}.
In seminal work, Koditschek showed that arbitrary first-order dynamics  \eqref{eq:first-order-autonomous-ode} can indeed be ``embedded'' within any ``fully actuated'' second-order system, at least in the classical Euclidean space context \cite[sec.~3.3]{koditschek1987adaptive}.
However, he also raised the important concern that the closed-loop second-order system might not asymptotically reproduce dynamical properties of the prescribed first-order system.  

The first main result of this paper, Theorem~\ref{th:naim-stabilization}, is based on a coordinate-free version of  the same \concept{Koditschek construction}.
It implies in particular that, if $\cfg$ is compact, $\pi_{\ctrl}$ is a smooth vector bundle, and $\cs$ is a smooth fiberwise surjective  linear map (see Proposition~\ref{prop:suff-cond-assump:naim-stabilization}), then the image $\vo(\cfg)\subset \T\cfg$ of an arbitrary smooth vector field $\vo$ on $\cfg$ can be made a globally exponentially stable and normally attracting invariant manifold (Definition~\ref{def:naim-bunching}) for the closed-loop dynamics on $\T \cfg$ resulting from substituting some smooth feedback law $\uc\colon \T \cfg \to \ctrl$ into \eqref{eq:riemannian-cs}.
Normal attractivity implies the crucial property that, not only does each closed-loop trajectory $t\mapsto \dot{\q}(t)$ converge to $\vo(\cfg)$, but also $t\mapsto \q(t)$ converges to a trajectory of \eqref{eq:first-order-autonomous-ode} at an exponential rate \cite[Thm~2]{fenichel1974asymptotic}, \cite[Thm~4.1]{hirsch1977invariant}.
Thus, the closed-loop second-order system asymptotically reproduces the prescribed first-order reference dynamics \eqref{eq:first-order-autonomous-ode}, resolving the aforementioned concern in our setting (see Remark~\ref{rem:koditschek}).
See Figure~\ref{fig:naim-stabilization} for an illustration.

While the first main result provides a positive conclusion for fully actuated systems, the second main result of this paper, Theorem~\ref{th:non-stabilizability}, provides a negative conclusion for underactuated systems.
It asserts in particular that, if $\cfg$ is compact with nonzero Euler characteristic and the image $\cs(\ctrl)\subset \T \cfg$ of $\cs$ is contained in a continuous regular distribution (tangent subbundle) of positive corank, then the image $\vo(\cfg)$ of ``almost every'' smooth vector field $\vo$ on $\cfg$ cannot even be made locally asymptotically stable by any feedback law $\uc\colon \T \cfg \to \ctrl$ such that $\cs\circ \uc$ is smooth or even locally Lipschitz.
Furthermore, the image $\cs(\ctrl)$ actually need only be contained in a continuous regular distribution over a small neighborhood of the zeros of $\vo$,  broadening the applicability of Theorem~\ref{th:non-stabilizability}.\footnote{For example, there are no globally-defined continuous regular distributions of positive corank on even-dimensional spheres \cite[problem~9-C]{milnor1974characteristic}, but Theorem~\ref{th:non-stabilizability} can still be applied to second-order control systems \eqref{eq:riemannian-cs}, \eqref{eq:F-cd} on such spheres by considering locally-defined distributions, as in Example~\ref{ex:non-stabilization}.\label{foot:local-underactuation-distributions}}
The result applies to a class of vector fields $\vo$ including all locally Lipschitz vector fields having only finitely many zeros (see Remark~\ref{rem:finite-zeros}).
(The set of all $C^k$ such vector fields contains an open and dense subset of the space of all $C^k$ vector fields on $\cfg$ with the $C^k$ topology for any $k\in \N_{\geq 1}\cup\{\infty\}$ \cite[Thm~2.3.2]{palis1982geometric}.)
Note that the Euler characteristic assumption is indeed necessary  (see Remark~\ref{rem:zero-euler-char-case} and Example~\ref{ex:euler-char-necessary}).

The proof of the second main result is based on a generalization and strengthening of Brockett's necessary condition for stabilizability of points in Euclidean spaces \cite[Thm~1.(iii)]{brockett1983asymptotic} to a necessary condition for stabilizability of general compact  sets in manifolds \cite[Thm~3.2]{kvalheim2022necessary} (see also \cite[Thm~1]{kvalheim2025relationships}).
See \cite{jongeneel2023topological} for a broad survey of such topological obstructions to stability and stabilization, and \cite{kvalheim2025relationships} for a narrower review.

The remainder of this paper is organized as follows.
Preliminary definitions are in section~\ref{sec:preliminaries}.
The (``positive'') first main result is stated and proved in section~\ref{sec:fully-actuated}.
The (``negative'') second main result is stated in section~\ref{sec:underactuated} and proved in section~\ref{sec:proof-of-th:non-stabilizability}.
A concluding discussion is in section~\ref{sec:discussion}.

\section{Preliminaries}\label{sec:preliminaries}
A \concept{section} of a map $\pi\colon A\to B$ is a right inverse of $\pi$, that is, a map $\sigma \colon B\to A$ satisfying $\pi\circ \sigma = \id_B$.
If $A$ and $B$ are topological spaces, we let $\sect^0(\pi)$ denote the set of $C^0$ (continuous) sections of $\pi$.
If $A$ and $B$ are smooth $(C^\infty)$ manifolds, we let $\sect^k(\pi)$ denote the set of $C^k$ sections of $\pi$ for $k\in \N_{\geq 0}\cup \{\infty\}$,  where a map is $C^k$ if it has $C^0$ partial derivatives of all orders less than $k+1$ in local coordinates.

Given any smooth manifold $M$, we let $\pi_{\T M}\colon \T M \to M$ denote the tangent bundle projection of $M$.
Given $V\subset M$ and $v\in V$, we let $\T_V M= \pi_{\T M}^{-1}(V)$ and $\T_v M= \T_{\{v\}}M$.
We let $\vf^k(M) = \sect^k(\pi_{\T M})$ denote the set of $C^k$ vector fields on $M$ for $k\in \N_{\geq 0}\cup \{\infty\}$.
Given a $C^1$ map $g\colon M\to W$ into another manifold $W$,  we denote by $\T g\colon \T M\to \T W$ the tangent map \cite[p.~18]{hirsch1994differential},  $\T_V g= \T g|_{\T_V M}\colon \T_V M \to \T_{g(V)} W$, and $\T_v g = \T_{\{v\}}g=\T g|_{\T_v M}\colon \T_vM\to \T_{g(v)}W$.

Throughout this paper, the \concept{configuration space} $\cfg$ is a compact smooth manifold with smooth Riemannian metric $\rmo$, and $\nabla$ is the associated Levi-Civita connection.
(However, less differentiability would suffice.)
%(If $k\in \N_{\geq 1}$ and $\cfg$ is a $C^k$ manifold, then $\cfg$ is $C^k$ diffeomorphic to a smooth manifold \cite[Thm~2.2.9]{hirsch1994differential}, so extra generality is available ``for free'' \MK{But what about the metric?}.)
The \concept{control space} $\ctrl$ is a topological space and the \concept{control space projection} $\pi_{\ctrl}\colon \ctrl \to \T\cfg$ is a $C^0$ map.
Given any subset $V\subset \T\cfg$ and point $v\in \T\cfg$, we let $\ctrl_V = \pi_{\ctrl}^{-1}(V)$ and $\ctrl_v = \ctrl_{\{v\}}$.

In addition to vector fields on $\cfg$, we will need to consider vector fields on $\T\cfg$ such as the \concept{geodesic spray} $\spray\in \vf^\infty(\T\cfg)$ of the Riemannian metric $\rmo$ \cite[sec.~22.6]{michor2008topics}, which is defined by the property that geodesics $\R\to \cfg$ for $\rmo$ are $\pi_{\T\cfg}$-projections of trajectories $\R\to \T\cfg$ of $\spray$.
Moreover, in addition to $\T\cfg$, we will need to consider the tangent bundle $\T^2 \cfg = \T(\T\cfg)$ of $\T\cfg$ with tangent bundle projection $\pi_{\T^2\cfg}\colon \T^2 \cfg \to \T \cfg$.
We denote by $\V(\T\cfg)$ the \concept{vertical bundle} $\ker \T {\pi_{\T\cfg}}\subset \T^2 \cfg$, 
\begin{equation*}\label{eq:fiber-product-version-of-whitney-sum}
	\T\cfg \times_{\cfg} \T\cfg =\{(v,w)\in \T\cfg \times \T\cfg\colon \pi_{\T\cfg}(v)=\pi_{\T\cfg}(w)\}
\end{equation*}
the fiber product of $\pi_{\T\cfg}$ with itself (equivalently, the pullback of $\pi_{\T\cfg}$ along itself, which coincides with the Whitney sum $\T\cfg\oplus \T\cfg$ as a manifold but not as a vector bundle), and  $\vl\colon \T\cfg \times_{\cfg} \T\cfg \to \V(\T\cfg)$ the \concept{vertical lift} given by 
\begin{equation}\label{eq:vertical-lift-defn}
\vl(v,w)=\frac{d}{dt}(v+tw)|_{t=0}
\end{equation}
for all $(v,w)\in \T\cfg \times_\cfg \T\cfg$ \cite[p.~107]{michor2008topics}.

\section{Stabilizability in the fully actuated case}\label{sec:fully-actuated}
\subsection{Normally attracting invariant manifolds}\label{subsec:fully-actuated-prelims}

In this subsection, we give precise definitions of $k$-normal attractivity and $k$-center bunching.
The former corresponds  to \cite[sec.~2.1]{eldering2018global}, \cite[Prop.~2.3]{hirsch1977invariant}, which is a natural version for flows of the classical definition of eventual relative $k$-normal hyperbolicity for diffeomorphisms in the special case of a zero unstable bundle \cite[sec.~1]{hirsch1977invariant} (see also \cite{fenichel1971persistence,eldering2013normally}).
The latter corresponds to a hypothesis in \cite[Cor.~2]{eldering2018global}, which in turn corresponds to a hypothesis in \cite[Thm~5]{fenichel1977asymptotic} reformulated in notation closer to \cite[Thm~B]{pugh1997holder} in the case $k=1$.
We first recall the ingredients.

Given a $C^1$ embedded submanifold $V$ of $M$ \cite[pp.~13, 21]{hirsch1994differential}, we let $\Nor V =  \T_V M / \T V$ denote the algebraic normal bundle of $V$ in $M$ \cite[p.~96]{hirsch1994differential} and  $\Nor_v V= \T_v M / \T_v V\subset \Nor V$ for $v\in V$.
If $g\colon M\to M$ is a $C^1$ map satisfying $g(V)\subset V$, then $\T g(\T V)\subset \T V$, so there is a canonical fiberwise linear map $\Nor  g\colon \Nor V \to \Nor V$.

Given a fiberwise linear map $L\colon E_1\to E_2$ between topological vector bundles $\pi^i \colon E^i\to B^i$ with $C^0$ fiberwise norms $\|\slot\|$ \cite[sec.~4.2]{hirsch1994differential}, let $E^i_b= (\pi^i)^{-1}(b)$ and
\begin{equation*}
	\norm{L|_{E^1_b}} = \max_{\norm{e}=1} \norm{L|_{E^1_b}(e)} \quad \text{and} \quad \minnorm{L|_{E^1_b}} = \min_{\norm{e}=1} \norm{L|_{E^1_b}(e)}
\end{equation*}
denote the induced norm and conorm of $L|_{E^1_b}$.

\begin{Def}\label{def:naim-bunching}
Let $V$ be a compact $C^1$ invariant manifold for a $C^1$ flow $\Phi$ on a smooth manifold $M$, and let $\Psi$ be the flow on $\Nor V $ given by $\Psi^t = \Nor \Phi^t$.
Fix any $C^0$ fiberwise norms $\|\slot\|$ on $\T V$ and $\Nor V$.
Given $k\in \N_{\geq 1}$, we say that $V$ is a \concept{$k$-normally attracting invariant manifold ($k$-NAIM)} if there are constants $a,b>0$ such that, for all $v\in V$ and integers $0\leq i \leq k$,
\begin{equation*}\label{eq:naim-cond}
\|\Psi^t|_{\Nor_v V}\| \leq b e^{-at}\minnorm{\T\Phi^t|_{\T_v V}}^i \quad \text{for all } t\geq 0.
\end{equation*}  
We say that $V$ is a \concept{normally attracting invariant manifold (NAIM)} if $V$ is a $1$-NAIM.
We say that a NAIM $V\subset M$ is \concept{$k$-center bunching} if there are constants $a,b>0$ such that, for all $v\in V$ and integers $0\leq i \leq k$,
\begin{equation*}\label{eq:center-bunching-cond}
	\norm{\T_v \Phi^t|_{\T_v V}}^i\|\Psi^t|_{\Nor_v V}\| \leq b e^{-at}\minnorm{\T_v \Phi^t|_{\T_v V}} \quad \text{for all } t\geq 0.
\end{equation*}  
\end{Def}

\begin{Rem}\label{rem:independence-of-norms}
	Any pair of $C^0$ fiberwise norms on a vector bundle over a compact base are uniformly equivalent, so Definition~\ref{def:naim-bunching} does not depend on the choice of norms.
\end{Rem}

\begin{Rem}\label{rem:center-bunching-open-condition}
If $k\in \N_{\geq 1}$,  $V$ is a $k$-NAIM for a $C^k$ flow $\Phi$, and $\Theta$ is another flow sufficiently close to $\Phi$ in the compact-open (weak) $C^k$ topology \cite[sec.~2.1]{hirsch1994differential}, then a classical fundamental result asserts that $\Theta$ has a unique $k$-NAIM $V'$ that is $C^k$-close to $V$ \cite[Thm~4.1(f)]{hirsch1977invariant}.	
Moreover, it is straightforward to show that $V$ is $k$-center bunching for $\Phi$ if and only if there is  $\tau>0$ such that, for all $v\in V$ and integers $0\leq i \leq k$,
\begin{equation}\label{eq:center-bunching-finite-time}
	\norm{\T_v \Phi^\tau|_{\T_v V}}^i\|\Psi^\tau|_{\Nor_v V}\| < \frac{1}{2}\minnorm{\T_v \Phi^\tau|_{\T_v V}},
\end{equation}  
and similarly for $V'$ and $\Theta$.
This  is clearly a $C^1$-open condition on flows, so the $k$-NAIM $V'$ for $\Theta$ will also be $k$-center bunching if $\Theta$ is sufficiently $C^k$-close to $\Phi$.
\end{Rem}

\subsection{Results}\label{subsec:fully-actuated-results}
Theorem~\ref{th:naim-stabilization} incorporates the following assumption.

\begin{Assump}\label{assump:naim-stabilization}
For any $C^k$ map $G\colon \T\cfg \to \T \cfg$ satisfying $G\circ \pi_{\T\cfg} = \pi_{\T\cfg}$, there is $\uc\in \sect^0(\pi_{\ctrl})$ such that $\cs\circ \uc = G$.
\end{Assump}

In particular, Assumption~\ref{assump:naim-stabilization} holds if $\pi_{\ctrl}$ is isomorphic to the pullback of the bundle $\pi_{\T \cfg}$ along itself \cite[p.~97]{hirsch1994differential}.
However, a more general sufficient condition is the following, where for any $\q\in \cfg$ and $v\in \T_\q \cfg$,  we let $\cs_v$ denote the restriction $\cs|_{\ctrl_v}\colon \ctrl_v\to \T_{\q}\cfg$, and $\cs_v^*, \cs_v^\dagger\colon \T_\q \cfg \to \ctrl_v$ respectively denote the adjoint (transpose) and Moore-Penrose pseudoinverse of $\cs_v$ with respect to the Riemannian metric $\rmo$ and some choice of fiberwise inner product on $\ctrl$ \cite[p.~95]{hirsch1994differential}.

\begin{Prop}\label{prop:suff-cond-assump:naim-stabilization}
If $\pi_{\ctrl}$ is a $C^k$ vector bundle and $\cs$ is a $C^k$ fiberwise surjective   linear map, then Assumption~\ref{assump:naim-stabilization} holds with $\uc \in \sect^k(\pi_{\ctrl})$ given by
\begin{equation}\label{eq:pseudo-inverse-formula-for-feedback}
\uc(v) = \cs_v^\dagger (G(v)) = \cs_v^* \circ (\cs_v \circ \cs_v^*)^{-1}(G(v))
\end{equation}
with respect to the metric $\rmo$ and any given $C^k$ fiberwise inner product on $\ctrl$.
\end{Prop}

\begin{proof}
The rightmost expression in \eqref{eq:pseudo-inverse-formula-for-feedback}	is clearly $C^k$ in $v$, and $\pi_{\ctrl}\circ \uc=\id_{\T\cfg}$ since $ \cs_v^* \circ (\cs_v \circ \cs_v^*)^{-1}$ is a map $\T_{\pi_{\T\cfg}(v)} \cfg\to \ctrl_v$, so $\uc\in \sect^k(\pi_{\ctrl})$.
And for any $v\in \T\cfg$, 
\begin{equation*}
	\begin{split}
		\cs\circ \uc(v)	&=  (\cs_v \circ \cs_v^*) \circ (\cs_v \circ \cs_v^*)^{-1}(G(v))\\
		&= G(v).
	\end{split}
\end{equation*}
\end{proof}

Proposition~\ref{prop:suff-cond-assump:naim-stabilization} and Remark~\ref{rem:no-redundant-expressions} should be kept in mind for the following statement of the first main result, a stabilization result with additional performance guarantees for fully actuated systems.\footnote{Here $\rmo^\flat, \rmt^\flat\colon \T\cfg \to \T^*\cfg$ and $\rmo^\sharp, \rmt^\sharp\colon \T^*\cfg\to \T\cfg$ are the standard \concept{musical isomorphisms} with respect to the  metrics $\rmo$ and $\rmt$ \cite[p.~128]{abraham1978foundations}.}
In particular, when the hypotheses of Proposition~\ref{prop:suff-cond-assump:naim-stabilization} hold, it provides the explicit feedback law \eqref{eq:pseudo-inverse-formula-for-feedback} satisfying \eqref{eq:kod-feedback}.
\begin{Th}\label{th:naim-stabilization}
Fix $k\in \N_{\geq 1}$ and let $\cs$ satisfy Assumption~\ref{assump:naim-stabilization}.
Then for any $\vo\in \vf^k(\cfg)$, $C^k$ Riemannian metric $\rmt$ on $\cfg$, and $\varepsilon > 0$, the closed-loop $C^k$ dynamics determined by
\begin{equation}\label{eq:riemannian-cs-closed-loop-epsilon}
		\nabla_{\dot{\q}}\dot{\q} = \cs\circ\uc_\varepsilon(\dot{\q})
\end{equation}
for any feedback law $u_\varepsilon\in \sect^0(\pi_{\ctrl})$ satisfying
\begin{equation}\label{eq:kod-feedback}
F\circ u_\varepsilon(v)= \nabla_v\vo - \frac{1}{\varepsilon}\rmo^\sharp\circ \rmt^\flat(v-\vo\circ \pi_{\T\cfg}(v))
\end{equation}
are complete and have $\vo(\cfg)\subset \T\cfg$ as a globally exponentially stable invariant manifold.
Moreover, $\vo(\cfg)$ is a $k$-NAIM and  $k$-center bunching for all sufficiently small $\varepsilon > 0$.
\end{Th}

\begin{Rem}\label{rem:koditschek}
Remarkably, the feedback law \eqref{eq:kod-feedback} is simply a coordinate-free version of the classical Koditschek construction for Euclidean spaces \cite[eq.~(20)]{koditschek1987adaptive}.
After introducing his construction, Koditschek articulated some important concerns (here $\mathcal{R}$ corresponds to our $\vo(\cfg)$) \cite[sec.~3.3]{koditschek1987adaptive}:
\begin{quote}
	``Unfortunately, this result is much too weak to be useful. 
	To begin with there is the assumption of completeness---i.e. that no finite escape trajectories result from the application of the control law.
	More fundamentally, the mere guarantee that $\mathcal{R}$ is a globally attracting set does not afford the conclusion that the behavior of trajectories starting away from $\mathcal{R}$ end up behaving at all like those which do originate in $\mathcal{R}$.
	In particular, if...[the dynamics]...is known to induce a positive limit set $\Omega\subset \mathcal{R}$, it is by no means clear that $\Omega$ is the positive limit set of the flow on the entire space.''
\end{quote}
Theorem~\ref{th:naim-stabilization} completely addresses these concerns in our setting.
To explain, let $\Phi_\varepsilon$ be the $C^k$ flow on $\T\cfg$ determined by \eqref{eq:riemannian-cs-closed-loop-epsilon}, \eqref{eq:kod-feedback} and $\Phi_\vo$ be the $C^k$ flow of $\vo$ on $Q$.
First, Theorem~\ref{th:naim-stabilization} asserts that $\Phi_\varepsilon$ is (both forward and backward) complete for any $\varepsilon > 0$, resolving the completeness concern.
Next, let $\varepsilon > 0$ be small enough as in the final sentence of Theorem~\ref{th:naim-stabilization}.
Then by \cite[Cor.~2]{eldering2018global}, there is a $C^{k-1}$ diffeomorphism  $H_\varepsilon\colon \T\cfg \to \T \cfg$ conjugating $\Phi_\varepsilon$ to a $C^{k-1}$ flow $\Theta_\varepsilon$ on $\T\cfg$ covering $\Phi_\vo$ such that $0_{\T\cfg}$ is globally exponentially stable for $\Theta_\varepsilon$.
In particular, 
\begin{equation*}
\begin{tikzcd}
	\T\cfg \arrow[r,"H_\varepsilon"] & \T\cfg \arrow[r,"\pi_{\T \cfg}"] & \cfg\\
	\T \cfg \arrow[r,"H_\varepsilon",swap] \arrow[u,"\Phi_\varepsilon^t"] & \T \cfg \arrow[r,"\pi_{\T \cfg}",swap] \arrow[u,"\Theta_\varepsilon^t"] & \cfg \arrow[u,"\Phi_X^t",swap]
\end{tikzcd}
\end{equation*}
is a commutative diagram for all $t\in \R$.
\iffalse
 fiber bundle $P\colon \T \cfg \to \cfg$ such that $P(\vo(\q))=\q$ for all $q\in \cfg$ and the diagram
\begin{equation*}\label{eq:asy-phase}
	\begin{tikzcd}
		\T \cfg \arrow[d,"P"] \arrow[r,"\Phi_\varepsilon^t"] & \T \cfg \arrow[d,"P"]\\
		\cfg \arrow[r,"\Phi_\vo^t",swap] & \cfg 
	\end{tikzcd}
\end{equation*}
commutes for all $t\in \R$ \cite[Cor.~2]{eldering2018global}.
\fi
It follows readily that, with respect to the natural diffeomorphic identification $\vo\colon \cfg\to \vo(\cfg)$, each $\Phi_\varepsilon$-trajectory asymptotically coalesces with some $\Phi_\vo$-trajectory at an exponential rate.
This implies that all positive limit sets and, e.g., the chain recurrent sets of $\Phi_\vo$ and $\Phi_\varepsilon$ coincide under the natural identification.
This is illustrated concretely by Example~\ref{ex:naim-stabilization} and Figure~\ref{fig:naim-stabilization}.
Of course, we have the technical advantage that $\cfg$ is compact, unlike in  Koditschek's Euclidean space setting.
\end{Rem}

\begin{proof}
First, since $\cfg$ is compact, $v\mapsto \rmt(v,v)$ attains a minimum $c > 0$ and maximum $C > 0$ on the $\rmo$-unit sphere bundle in $\T \cfg$, so for all $v\in \T\cfg$ we have
\begin{equation}\label{eq:norm-equivalence}
c\rmo(v,v) \leq \rmt(v,v) \leq C \rmo(v,v).
\end{equation}	
	
Let $t\mapsto \q(t)\in \cfg$ be any solution of \eqref{eq:riemannian-cs-closed-loop-epsilon}, \eqref{eq:kod-feedback}.
Define $y(t)=\dot{\q}(t)-\vo(\q(t))$.
Then \eqref{eq:kod-feedback} implies that
\begin{equation}\label{eq:nabla-qdot-y}
	\begin{split}
	\nabla_{\dot{\q}}y &= \cs \circ \uc_\varepsilon(\dot{\q})-\nabla_{\dot{\q}}\vo = - \frac{1}{\varepsilon}\rmo^\sharp\circ \rmt^\flat(\dot{\q}-\vo(q))\\
	&= - \frac{1}{\varepsilon}\rmo^\sharp\circ \rmt^\flat(y),
\end{split}
\end{equation}
so differentiating $t\mapsto \rmo(y(t),y(t))$ yields 
\begin{align*}
	\frac{d}{dt}\rmo(y,y)&= 2\rmo(\nabla_{\dot{\q}}y,y)= - \frac{2}{\varepsilon}\rmo(\rmo^\sharp\circ \rmt^\flat(y),y )= -\frac{2}{\varepsilon} \rmt(y,y).
\end{align*}
This and \eqref{eq:norm-equivalence} yield
\begin{equation*}
 -\frac{2C}{\varepsilon} \rmo(y,y) \leq	\frac{d}{dt}\rmo(y,y)  \leq -\frac{2c}{\varepsilon} \rmo(y,y),
\end{equation*}
so for all $t$ in the domain of $t\mapsto \dot{q}(t)$ we have 
\begin{equation*}
e^{-(2C/\varepsilon)t}\rmo(y_0,y_0)\leq 	\rmo(y(t),y(t))\leq e^{-(2c/\varepsilon)t}\rmo(y_0,y_0),
\end{equation*}
where $y_0=y(0)$.
Since $v\mapsto \rmo(v,v)$ has compact sublevel sets, it follows that the dynamics of \eqref{eq:riemannian-cs-closed-loop-epsilon}, \eqref{eq:kod-feedback} are complete and  $\vo(\cfg)$ is a globally exponentially stable invariant manifold.

The dynamics of \eqref{eq:riemannian-cs-closed-loop-epsilon}, \eqref{eq:kod-feedback} is the flow of the vector field $\vt_\varepsilon\in \vo^k(\T\cfg)$ given by
\begin{equation*}
	\vt_\varepsilon(v) = \spray(v) + \vl(v,\nabla_v \vo) - \frac{1}{\varepsilon}\vl(v,\rmo^\sharp\circ \rmt^\flat(v-\vo\circ \pi_{\T\cfg}(v))),
\end{equation*}
where $\spray$ is the geodesic spray of $\rmo$ and $\vl$ is  the vertical lift (section~\ref{sec:preliminaries}). 

Define a ``regularized'' vector field $\tilde{\vt}_\varepsilon = \varepsilon \vt_\varepsilon \in \vo^k(\T\cfg)$.
This has the nice property that, while $(\varepsilon, v)\mapsto \vt_\varepsilon(v)$ is well-defined and $C^k$ only on $(\R\setminus \{0\})\times \T\cfg$, $(\varepsilon, v)\mapsto \tilde{\vt}_\varepsilon(v)$ is well-defined and $C^k$ on all of $\R\times \T\cfg$.
Moreover, trajectories $t\mapsto v(t)$ of $\tilde{\vt}_\varepsilon$ satisfy $v(t)=\dot{q}(\varepsilon t)$, where $t\mapsto \dot{q}(t)$ is the  trajectory of $\vt_\varepsilon$ satisfying $\dot{q}(0)=v(0)$.
Thus, for any $\varepsilon > 0$, $\vo(\cfg)$ is a $k$-NAIM and $k$-center bunching for $\vt_\varepsilon$ if and only if it is the same for $\tilde{\vt}_\varepsilon$.

Note that $\vo(\cfg)$ is the set of zeros of the vector field 
\begin{equation*}
	\tilde{\vt}_0(v) = -\vl(v,\rmo^\sharp\circ \rmt^\flat(v-\vo\circ \pi_{\T\cfg}(v))).
\end{equation*}
Since $\rmo,\rmt$ are given locally by symmetric positive definite matrices, it readily follows that the linearization of $\tilde{\vt}_0$ at any such zero has $\dim \cfg$ zero eigenvalues and $\dim \cfg$ negative real eigenvalues.
Thus, $\vo(\cfg)$ is a $k$-NAIM and $k$-center bunching for $\tilde{\vt}_0$. %In fact $\infty$-NAIM and $\infty$-center bunching, but I don't need that, and those aren't open conditions anyway.
Since $\tilde{\vt}_\varepsilon \to \tilde{\vt}_0$ in the compact-open $C^k$ topology as $\varepsilon \to 0$, it follows that, for all sufficiently small $\varepsilon > 0$, $\vo(\cfg)$ is a $k$-NAIM and $k$-center bunching (see Remark~\ref{rem:center-bunching-open-condition}) for the flow of $\tilde{\vt}_\varepsilon$ and hence also the flow of $\vt_\varepsilon$, that is, the dynamics determined by \eqref{eq:riemannian-cs-closed-loop-epsilon}, \eqref{eq:kod-feedback}.
\end{proof}

\iffalse
\begin{Rem}\label{rem:linearizability}
	We close this section by noting a connection with linearizing, or ``Koopman'', embeddings of nonlinear systems \cite{liu2023non,belabbas2023sufficient,arathoon2023koopman,ko2024minimum,liu2025properties,kvalheim2025global}.
	In the notation of Remark~\ref{rem:koditschek}, $\vo \circ \pi_{\T\cfg}\circ H_\varepsilon$ is a $C^{k-1}$ \concept{asymptotic phase map} in the sense of \cite[sec.~2.3--2.4]{kvalheim2023linearizability}.
	By \cite[Thm~4]{kvalheim2023linearizability}, the existence of a $C^{k-1}$ asymptotic phase map is necessary for the existence of a $C^{k-1}$ embedding $h\colon \T\cfg\to \R^n$ that linearizes the dynamics of \eqref{eq:riemannian-cs-closed-loop-epsilon}, \eqref{eq:kod-feedback}, i.e., $y=h(\dot{q})$ satisfies
	$$\dot{y}=Ay$$
	for some $n\in \N$ and $A\in \R^{n\times n}$. 
	Moreover, the existence of a $C^0$ asymptotic phase map is necessary even for the existence of an injective $C^0$ map that  linearizes the dynamics, and any such linearizing map is automatically a proper $C^0$ embedding \cite[Thm~3]{kvalheim2023linearizability}.
\end{Rem}
\fi

\begin{Ex}\label{ex:naim-stabilization}
Let $\cfg$ be the unit sphere $S^2\subset \R^3$ with Riemannian metric $\rmo = \rmt$ given by the round metric, namely, the restriction of the Euclidean inner product $\ip{\slot}{\slot}$ to $T S^2$.
Viewing vector fields on $S^2$ as maps $S^2\to \R^3$ with the standard constraint \cite[p.~50]{guillemin1974differential}, let  $\vo\in \vf^\infty(S^2)$ be given by 
\begin{equation}\label{eq:ex:naim-stabilization-X-def}
	X(q_1,q_2,q_3) = (-q_2, q_1,0).
\end{equation}
Let $\Pi\colon \T_{S^2}\R^3\to \T S^2$ be the fiberwise orthogonal projection 
\begin{equation*}%\label{eq:ex:naim-stabilization-Pi-def}
	\Pi(w) = w-\ip{w}{q}q
\end{equation*}
for $q\in S^2$ and $w\in T_q \R^3$.
Then the Levi-Civita connection $\nabla$ for $\rmo$ satisfies
\begin{equation*}%\label{eq:ex:naim-stabilization-nabla-v-X}
	\begin{split}
	\nabla_v X &= \Pi(\T X(v)) \\
	&= D_q X \cdot v - \ip{D_q X\cdot v}{q}q
	\end{split}
\end{equation*}
for $q\in S^2$ and $v\in T_q S^2$, where $D_q X \in \R^{3\times 3}$ is the Jacobian matrix of the map $\R^3\to \R^3$ defined by the right side of \eqref{eq:ex:naim-stabilization-X-def}, ``$\cdot$'' denotes matrix multiplication, and $q,v$ are viewed as column vectors.
Similarly,
\begin{equation*}%\label{eq:ex:naim-stabilization-nabla-qdot-qdot}
	\begin{split}
		\nabla_{\dot{q}} \dot{q} &= \Pi(\ddot{q}) = \ddot{q}-\ip{\ddot{q}}{q}q\\
		& = \ddot{q} + \ip{\dot{q}}{\dot{q}}q
	\end{split}
\end{equation*} 
for any $C^2$ curve $t\mapsto q(t)\in S^2$, where the third equality follows from twice differentiating the constant curve $t\mapsto \ip{q(t)}{q(t)}\equiv 1$ using the product rule.
Thus, the closed-loop dynamics \eqref{eq:riemannian-cs-closed-loop-epsilon}, \eqref{eq:kod-feedback} of Theorem~\ref{th:naim-stabilization} take the concrete form
\begin{equation}\label{eq:ex:naim-stabilization-closed-loop}
	\begin{split}
		\ddot{q} &= D_q X \cdot \dot{q} - \ip{D_q X\cdot \dot{q}}{q}q- \ip{\dot{q}}{\dot{q}}q - \frac{1}{\varepsilon}(\dot{q}-X(q)).
	\end{split}
\end{equation}
Note that, since the first-order dynamics have no hyperbolicity, $\vo(\cfg)$ is a globally exponentially stable $\infty$-NAIM and  $\infty$-center bunching ($k$-NAIM and $k$-center bunching for every finite $k$) for the second-order system \eqref{eq:ex:naim-stabilization-closed-loop} for \emph{any}  $\varepsilon > 0$.
Taking $\varepsilon = 1.2$, Figure~\ref{fig:naim-stabilization} shows solutions $t\mapsto q(t)$ for several initial conditions $(q(0),\dot{q}(0))$ converging to corresponding trajectories of the first-order dynamics generated by the vector field  \eqref{eq:ex:naim-stabilization-X-def}, as guaranteed by normal attractivity (see Remark~\ref{rem:koditschek}).
\end{Ex}

\begin{figure}
	\centering
	\includegraphics[width=0.5\linewidth]{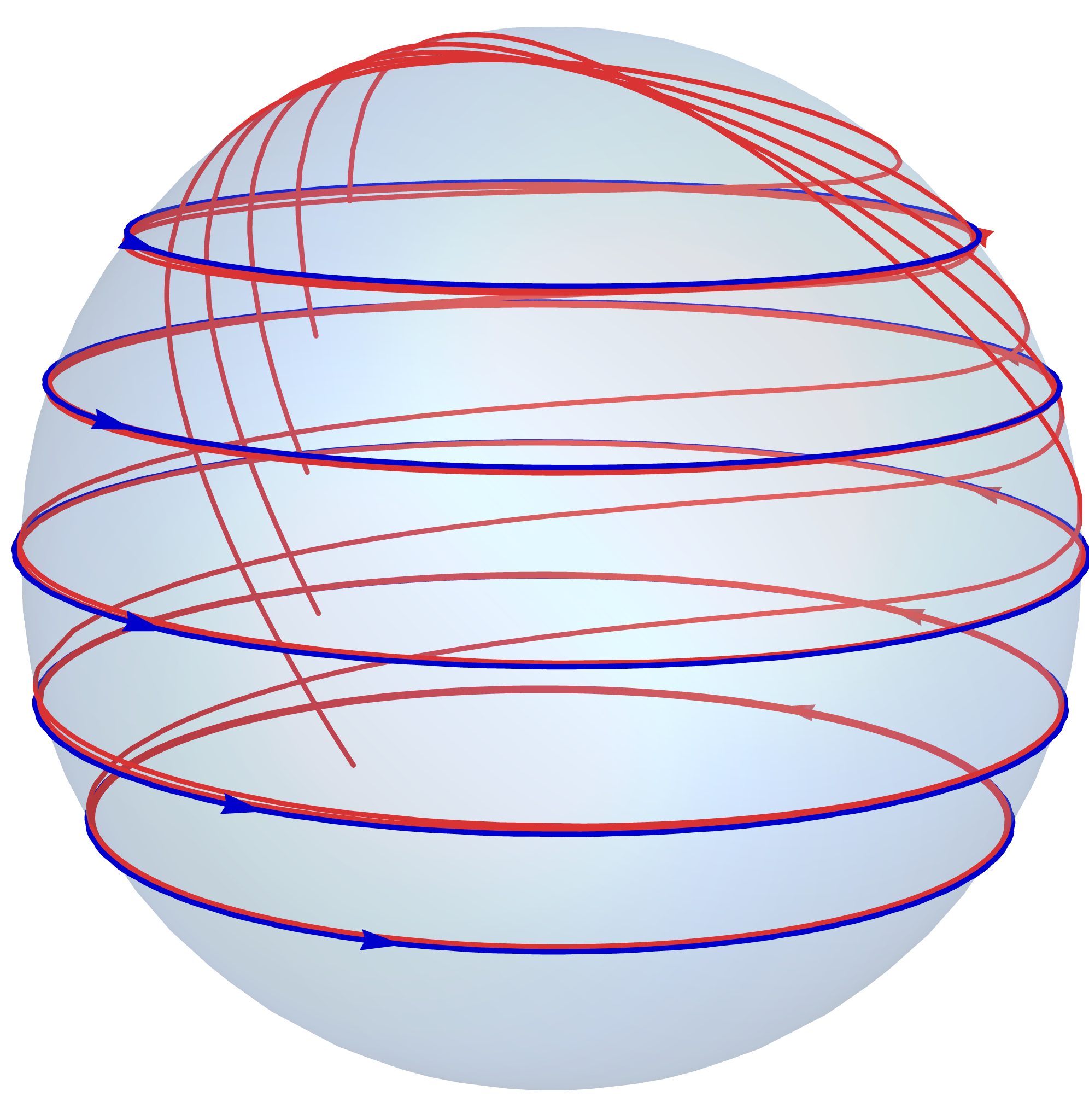}\caption{Illustration of Example~\ref{ex:naim-stabilization}. Shown in blue are several trajectories of the first-order reference dynamics generated by the vector field \eqref{eq:ex:naim-stabilization-X-def} on $\cfg = S^2$.
	Shown in red are several solutions to the closed-loop second-order system \eqref{eq:ex:naim-stabilization-closed-loop} for $\varepsilon = 1.2$.
	As explained in Example~\ref{ex:naim-stabilization}, here $\vo(S^2)\subset \T S^2$ is an $\infty$-NAIM and $\infty$-center bunching for \emph{any} value of $\varepsilon > 0$.
	Indeed, each red solution exponentially converges to a blue trajectory, so the closed-loop second-order system asymptotically reproduces the prescribed first-order reference dynamics, as guaranteed by NAIM theory (see Remark~\ref{rem:koditschek}).
	} \label{fig:naim-stabilization}
\end{figure}	

\section{Non-stabilizability in the underactuated case}\label{sec:underactuated}

Theorem~\ref{th:non-stabilizability} incorporates the following assumption.

\begin{Assump}\label{assump:non-stabilization}
There is a neighborhood $\nbhd\subset \cfg$ of $\vo^{-1}(0_{\T\cfg})$, a $C^0$ regular distribution $\dbn$ on $\nbhd$, and $\vt\in \vf^0(\nbhd)$ such that $\cs(\ctrl_{\T_\nbhd \cfg})\subset \dbn$ and $\vt(\nbhd)\cap \dbn = \varnothing$.
\end{Assump}

\begin{Rem}\label{rem:finite-zeros}
If $\vo^{-1}(0_{\T\cfg})$ is finite and there is a $C^0$ regular distribution $\dbn'$ of positive corank on a neighborhood $\nbhd'\subset \cfg$ of $\vo^{-1}(0_{\T\cfg})$ such that $\cs(\ctrl_{\T_{\nbhd'} \cfg})\subset \dbn'$, then Assumption~\ref{assump:non-stabilization} holds for some $\nbhd\subset \nbhd'$, $\vt\in \vf^0(\nbhd)$, and $\dbn=\dbn'|_W$.
\end{Rem}

Remarks~\ref{rem:no-redundant-expressions}, \ref{rem:finite-zeros} should be kept in mind for the following statement of the second main result, a non-stabilizability result for underactuated systems.
For the statement, we say that \eqref{eq:riemannian-cs-closed-loop} is \concept{uniquely integrable} if there is a unique maximal trajectory $t\mapsto \q(t)$ satisfying $\dot{\q}(0)=v_0$ for any $v_0\in \T\cfg$.

\begin{Rem}\label{rem:uniquely-integrable-second-order}
	By the Picard-Lindel\"of theorem, \eqref{eq:riemannian-cs-closed-loop} is automatically uniquely integrable if $\cs\circ \uc$ is locally Lipschitz (e.g., $\ctrl$ is a smooth manifold and $u$, $\cs$ are locally Lipschitz).
\end{Rem}

\begin{Th}\label{th:non-stabilizability}
Fix any $\vo\in \vf^0(\cfg)$.
Assume that the Euler characteristic of $\cfg$ is nonzero and Assumption~\ref{assump:non-stabilization} holds.
Then for any $\uc\in \sect^0(\pi_{\ctrl})$ such that 
\begin{equation}\label{eq:riemannian-cs-closed-loop}
		\nabla_{\dot{\q}}\dot{\q} = \cs\circ\uc(\dot{\q})
\end{equation}
is uniquely integrable, $\vo(\cfg)$ is not an asymptotically stable invariant set for the dynamics determined by \eqref{eq:riemannian-cs-closed-loop}.
\end{Th}

\begin{Rem}\label{rem:zero-euler-char-case}
The assumption that $\cfg$ has nonzero Euler characteristic is indeed necessary for  Theorem~\ref{th:non-stabilizability}, as demonstrated below in Example~\ref{ex:euler-char-necessary}.
Similarly, in forthcoming work we show that, if we are not in the simple situation of Remark~\ref{rem:finite-zeros}, then the existence of $\dbn$ without $\vt$ in Assumption~\ref{assump:non-stabilization} is not sufficient for the conclusion of Theorem~\ref{th:non-stabilizability} to hold.
\end{Rem}

\begin{Ex}\label{ex:non-stabilization}
	Let $\cfg = S^2\subset \R^3$, $\rmo= \rmt$, and $\vo$ be as in Example~\ref{ex:naim-stabilization}.
	However, viewing $\cs$ as an $\R^3$-valued map,  assume now that the first component of $\cs(\uc)$ is constrained to vanish whenever $\q\in S^2$ is near the north and south poles constituting the zero set of $\vo$, so $\cs$ takes values there in the locally-defined smooth regular distribution $\dbn$ %of corank $1$ 
	given by
	$$\dbn_\q = \{v\in \T_{\q}S^2\colon v_1 = 0\}.$$
	Then Assumption~\ref{assump:non-stabilization} holds by Remark~\ref{rem:finite-zeros}, so Theorem~\ref{th:non-stabilizability} implies that $\vo(S^2)\subset \T S^2$ cannot even be made locally asymptotically stable for \eqref{eq:riemannian-cs-closed-loop} for any locally Lipschitz feedback law $\uc \in \sect^0(\pi_{\ctrl})$, let alone globally exponentially stable and more as in Example~\ref{fig:naim-stabilization}.
\end{Ex}

\begin{Ex}\label{ex:euler-char-necessary}
Consider $\cfg = SO(3)\subset \R^{3\times 3}$ with Riemannian metric
\begin{equation}\label{eq:left-invariant-metric}
	\rmo(R \xi^\wedge, R \eta^\wedge) = \ip{\xi}{\mathbb{I} \eta}
\end{equation}
for $R\in SO(3)$ and $\xi,\eta\in \R^3$, where $\mathbb{I} = \text{diag}(I_1,I_2,I_3)\in \R^{3\times 3}$ with $I_1, I_2, I_3 > 0$ and  $(\cdot)^{\wedge}\colon \R^3\to \mathfrak{so}(3)$ denotes the  standard ``hat map'' Lie algebra isomorphism \cite[sec.~1.9]{bloch2015nonholonomic} with inverse $(\cdot)^\vee\colon \mathfrak{so}(3)\to \R^3$.
\iffalse
map
\begin{equation*}
	\xi^\wedge = \begin{pmatrix}
		0 & -\xi_3 & \xi_2 \\
		\xi_3 & 0 & -\xi_1\\
		-\xi_2 & \xi_1 & 0
	\end{pmatrix}.
\end{equation*}
\fi
Letting $\ctrl =SO(3) \times \R^2$, $\pi_{\ctrl}\colon SO(3)\times \R^2\to SO(3)$ be given by $\pi_{\ctrl}(R,u_1,u_2)= R$, and $\cs\colon SO(3)\times\R^2 \to \T SO(3)$ be given using the standard basis  $e_1,e_2,e_3\in \R^3$ by
\begin{equation*}
\cs(R,u_1,u_2) = u_1 R \cdot (\mathbb{I}^{-1}e_1)^{\wedge}  + u_2 R \cdot (\mathbb{I}^{-1}e_2)^\wedge ,
\end{equation*}
the control system \eqref{eq:riemannian-cs} specializes to 
\begin{equation}\label{eq:riemannian-cs-so(3)}
	\nabla_{\dot{R}}\dot{R} = u_1 R \cdot (\mathbb{I}^{-1}e_1)^{\wedge} + u_2 R \cdot (\mathbb{I}^{-1}e_2)^{\wedge}.
\end{equation}

Since the image of $\cs$ is a smooth regular distribution of corank $1$ and $SO(3)$ has Euler characteristic zero, all  hypotheses of Theorem~\ref{th:non-stabilizability} except for the Euler characteristic assumption are satisfied for any nonvanishing $\vo\in \vf^0(SO(3))$. 
However, if $I_1 \neq I_2$, then  the image $\vo(SO(3))\subset \T SO(3)$ of the smooth nonvanishing vector field $\vo(R) = Re_1^{\wedge}$ can be made globally asymptotically stable by substituting a  smooth feedback law $(R,\dot{R})\mapsto (u_1(R, \dot{R}), u_2(R, \dot{R}))$ into \eqref{eq:riemannian-cs-so(3)}, as we explain next.
Thus, the  Euler characteristic assumption is necessary for Theorem~\ref{th:non-stabilizability}. 

To explain, the covariant acceleration $\nabla_{\dot{R}}\dot{R}$ of any $C^2$ curve $t\mapsto R(t)\in SO(3)$ with body velocity $\Omega = (R^{-1}\dot{R})^\vee \in \R^3$ satisfies \cite[Prop.~4]{zefran1998generation}
\begin{equation*}
(R^{-1}\nabla_{\dot{R}}\dot{R})^\vee = \dot{\Omega} + \mathbb{I}^{-1}(\Omega \times (\mathbb{I}\Omega)),
\end{equation*}
so  \eqref{eq:riemannian-cs-so(3)} is equivalent to the   Euler equations for a rigid body with two ``control jets''
\begin{equation}\label{eq:euler-equations}
\begin{split}
\dot{R} &= R \Omega^\wedge\\
\dot{\mathbb{I}\Omega} &=  (\mathbb{I}\Omega)\times \Omega + u_1 e_1 + u_2 e_2
\end{split}
\end{equation}
on $SO(3)\times \R^3 \approx \T SO(3)$.
It is known that $SO(3)\times \{e_1\}$ can be made globally asymptotically stable by substituting a smooth feedback law $\Omega\mapsto (u_1(\Omega), u_2(\Omega))$ into \eqref{eq:euler-equations} if $I_1 \neq I_2$ \cite[sec.~3]{wan1993rotational}, so  $\vo(SO(3))$ is globally asymptotically stable for the closed-loop dynamics resulting from substituting the corresponding smooth feedback law  $(R,\dot{R})\mapsto (u_1(R^{-1}\dot{R}), u_2(R^{-1}\dot{R}))$ into \eqref{eq:riemannian-cs-so(3)}.
\end{Ex}

\iffalse
\begin{Rem}
\MK{ For case of nonzero Euler characteristic where $X = 0$ and hence in particular $Z=X^{-1}(0_{\T\cfg})=\cfg$  is not finite, recall LaSalle argument (see favorited March 18 2026 iPhone photo) for stabilizing zero section with $\vo(\ctrl)\subset \dbn$ for some globally-defined distribution $\dbn$.
It suffices to find a positive-rank $\dbn$ and $\rmo$ such that no $\rmo$ geodesic is everywhere tangent to $\dbn$.
 I think a promising example is $\cfg = S^2 \times S^2$ with $\dbn = \T S^2 \times \{0_{S^2}\}$ and a metric on $\cfg$ such that no geodesic is contained in any slice $S^2 \times \{0\}$.
 But I need to make sure I can construct this metric (not sure if hard?).} 
\end{Rem}
\fi

\section{Proof of Theorem~\ref{th:non-stabilizability}}\label{sec:proof-of-th:non-stabilizability}

\subsection{The adversary theorem}\label{subsec:underactuated-proof-prelims}
To prove Theorem~\ref{th:non-stabilizability}, we will use an extension of Brockett's necessary condition for stabilizability of points in Euclidean spaces \cite[Thm~1.(iii)]{brockett1983asymptotic} to one for stabilizability of general compact  sets in manifolds \cite[Thm~3.2]{kvalheim2022necessary} (see also \cite[Thm~1]{kvalheim2025relationships}).
In this subsection we recall the statement of this  result for the special case of stabilizability of submanifolds.

The result uses a general framework for first-order control systems.
It consists of a \concept{state space} given by a smooth manifold $\st$, a \concept{control space} given by a topological space $\ctrl$ and a $C^0$ map $\cp\colon \ctrl \to \st$, and a \concept{control system} given by a $C^0$ map $f\colon \ctrl \to \T \st$ making the diagram
\begin{equation}\label{eq:f-cd-general}
	\begin{tikzcd}
		\ctrl \arrow[rr,"f"] \arrow[dr,"\cp",swap]& & \T\st \arrow[dl,"\pi_{\T\st}"]\\
		& \st & 
	\end{tikzcd}
\end{equation}
commute. 
(We continue to use our global notations $\ctrl$ and $\pi_{\ctrl}$ in this subsection since no confusion will arise, as the only specific first-order control system considered later in this paper will have $\ctrl$ and $\pi_{\ctrl}$ coinciding with those in \eqref{eq:riemannian-cs}.)

A compact set $\att\subset \st$ is \concept{stabilizable} if there is a \concept{feedback law} $\uc\in \sect^0(\cp)$ for which the \concept{closed-loop vector field} $f\circ \uc$ is uniquely integrable and has $\att$ as an asymptotically stable invariant set. 
Here $f\circ \uc$ is \concept{uniquely integrable} if there is a unique maximal trajectory $t\mapsto x(t)$ satisfying $x(0)=x_0$ for any $x_0\in \st$.

\begin{Rem}\label{rem:uniquely-integrable-first-order}
	Similarly to Remark~\ref{rem:uniquely-integrable-second-order}, $f\circ \uc$ is automatically uniquely integrable if $f\circ u$ is locally Lipschitz (e.g., if $\ctrl$ is a smooth manifold and  $u$, $f$ are locally Lipschitz).
\end{Rem}

\begin{Rem}\label{rem:cont-sys-framework} 
	This framework reduces to the ``bundle picture of control'' \cite[p.~15]{brockett1977control}, \cite[sec.~4.6]{bloch2015nonholonomic} in the special case that $\cp\colon \ctrl \to \st$ is a fiber bundle and $f\colon \ctrl \to \T \st$ is sufficiently smooth.
	If $\cp$ is a trivial (product) bundle $\cp \colon \ctrl = \st \times E\to \st$, then elements of $\ctrl$ are of the form $(x,\uc)$, so the expression ``$f(x,\uc)$'' makes literal sense.
	But in the general framework here, it is more appropriate to instead write expressions such as ``$f(\uc)$'' for $\uc\in \ctrl$ and ``$f\circ \uc$'' for $\uc \in \sect^0(\cp)$, keeping in mind that $\uc\in \ctrl$ contains all information about $x=\cp(\uc)$ (cf. Rem.~\ref{rem:no-redundant-expressions}).
%	This is why the notation $f\circ \uc$ is employed here for feedback laws $\uc\colon \st\to \ctrl$.
\end{Rem}

The result below refers to the following reformulation \cite[sec.~1.4]{kvalheim2025relationships} more closely resembling  \cite[Thm~1.(iii)]{brockett1983asymptotic} of a condition from  \cite[Thm~3.2]{kvalheim2022necessary}.
The terminology is explained by interpreting it as the ability to ``defeat'' ``adversaries'' $\adv$ via the image of $f$ somewhere  intersecting that of $\adv$.

\begin{AdvCond}
For any neighborhood $\nbhd_\att\subset \st$ of $\att\subset \st$ there is a neighborhood $\nbhd_0 \subset \T \st$ of the zero section such that, for any $\nbhd_0$-valued vector field $\adv$ over $\nbhd_\att$, $\adv(x)=f(\uc_x)$ for some $x\in \nbhd_\att$, $\uc_x\in \ctrl_x$.
\end{AdvCond}

The following result is a special case of  \cite[Thm~3.2]{kvalheim2022necessary} (see also \cite[Thm~1]{kvalheim2025relationships}).

\begin{AdvThm}%\label{th:adversary-theorem}
	Let $\att\subset \st$ be homeomorphic to a compact smooth manifold with nonzero Euler characteristic.
    If $\att$ is stabilizable for \eqref{eq:f-cd-general}, then the adversary condition holds.	
\end{AdvThm}

Since $\att$ is homeomorphic to a compact smooth manifold, $\att$ admits a finite triangulation \cite[Thm~II.10.6]{munkres1966elementary} and  the Euler characteristic of $\att$ is the usual finite alternating sum $$\#(\text{vertices})-\#(\text{edges})+\#(\text{faces})-\ldots,$$
an integer. 
However, for the case of a general compact set $\att\subset \st$, the Euler characteristic instead needs to be defined using \v{C}ech(-Alexander-Spanier) cohomology for the more general version of the adversary theorem to hold.
See \cite[sec.~11.2]{kvalheim2025relationships} for a brief overview and \cite[sec.~2]{kvalheim2022necessary} for further details.

\subsection{Proof}\label{subsec:underactuated-proof-lemmas}
In this section we prove two lemmas and then Theorem~\ref{th:non-stabilizability}.
We will use a specific case of the general first-order control system \eqref{eq:f-cd-general} given by the commutative diagram
\begin{equation}\label{eq:f-cd-specific}
	\begin{tikzcd}
		\ctrl \arrow[rr,"f"] \arrow[dr,"\pi_{\ctrl}",swap]& & \T^2\cfg \arrow[dl,"\pi_{\T^2\cfg}"]\\
		& \T \cfg & 
	\end{tikzcd},
\end{equation}
where 
\begin{equation}\label{eq:f-specific-def}
f = \spray\circ \pi_\ctrl + \vl\circ (\pi_{\ctrl},\cs),
\end{equation}
$\spray$ is the geodesic spray of $\rmo$, and $\vl$ is  the vertical lift (section~\ref{sec:preliminaries}).
The following result explains the relationship of \eqref{eq:f-cd-specific}, \eqref{eq:f-specific-def} with the second-order control system \eqref{eq:riemannian-cs}, \eqref{eq:F-cd}.

\begin{Lem}\label{lem:equivalence-between-two-cs} 
	For any $\uc\in \sect^0(\pi_{\ctrl})$, a $C^2$ curve $t\mapsto \q(t)$ in $\cfg$ satisfies \eqref{eq:riemannian-cs-closed-loop} if and only if it satisfies
\begin{equation}\label{eq:non-riemannian-cs-closed-loop} 
	\ddot{\q}=f\circ \uc(\dot{\q}).
\end{equation}	
\end{Lem}

\begin{proof}
This is immediate from the fact that, for any $C^2$ curve $t\mapsto \q(t)$ in $\cfg$, 
\begin{equation*}\label{eq:q-ddot-decomp}
	\ddot{\q} = \spray(\dot{\q}) + \vl(\dot{\q},\nabla_{\dot{\q}}\dot{\q}).
\end{equation*}
\end{proof}

\begin{Lem}\label{lem:adversary-existence}
For any $\vo\in \vf^0(\cfg)$ satisfying Assumption~\ref{assump:non-stabilization}, there is a neighborhood $\nbhdt\subset \T\cfg$ of $\vo(\cfg)$ and $\vtt\in \vf(\T\cfg)$ such that, for all $\varepsilon \in (0,1)$, $\varepsilon \vtt(\nbhdt)\cap f(\ctrl_{\nbhdt})=\varnothing$.
\end{Lem}

\begin{proof}
Let $\nbhd$, $\dbn$, and $\vt$ be as in Assumption~\ref{assump:non-stabilization}.
We may assume that $\nbhd$ is open.
Let $\varphi\colon \cfg \to [0,\infty)$ be any smooth function such that $\varphi^{-1}(0) = \cfg \setminus \nbhd$.	
Define $\vtt_0, \vtt\in \vf^0(\T\cfg)$ by $$\vtt_0(v)=\varphi(\pi_{\T\cfg}(v))\vl(v,\vt(\pi_{\T\cfg}(v)))$$
and $\vtt = \spray + \vtt_0$.
Define $\nbhdt_0 = \pi_{\T \cfg}^{-1}(\nbhd)$, let $\nbhdt_1\subset \T\cfg$ be any open neighborhood of $\vo(\cfg)\setminus \nbhdt_0$ disjoint from $0_{\T\cfg}$, and define $\nbhdt = \nbhdt_0 \cup \nbhdt_1$.

For any $\varepsilon \in (0,1)$, $\varepsilon \spray$ is not equal to $\spray$ anywhere except on $0_{\T\cfg}$.
Since $\varepsilon \spray$ and $\spray\circ \pi_{\ctrl}$ are respectively the horizontal parts of $\vtt$ and $\spray\circ \pi_{\ctrl}$, it follows that
\begin{equation*}
\varepsilon \vtt(\nbhdt) \cap f(\ctrl_{\nbhdt}) \subset \nbhdt \cap 0_{\T\cfg} = \nbhdt_0 \cap 0_{\T \cfg}
\end{equation*}
for any $\varepsilon\in (0,1)$.
Thus, for $\varepsilon \in (0,1)$,
\begin{align*}
\varepsilon \vtt(\nbhdt) \cap f(\ctrl_{\nbhdt})&= \varepsilon \vtt_0(\nbhdt_0\cap 0_{\T\cfg}) \cap f(\ctrl)
\end{align*}
is empty if  $\varepsilon\vt(\nbhd) \cap \cs(\ctrl_{\T_\nbhd \cfg})$ is empty, which is the case since our assumptions imply that
\begin{align*}
\varepsilon\vt(\nbhd) \cap \cs(\ctrl_{\T_\nbhd \cfg})&\subset \varepsilon\vt(\nbhd) \cap	\dbn_\nbhd = \varnothing.
\end{align*}
\end{proof}	

We now prove the second main result.

\begin{proof}[Proof of Theorem~\ref{th:non-stabilizability}]
Fix any $\uc \in \sect^0(\pi_{\ctrl})$.
Lemma~\ref{lem:equivalence-between-two-cs}  implies that \eqref{eq:riemannian-cs-closed-loop} is uniquely integrable if and only if  \eqref{eq:non-riemannian-cs-closed-loop} is, and that $\vo(\cfg)$ is asymptotically stable for \eqref{eq:riemannian-cs-closed-loop} if and only if it is the same for \eqref{eq:non-riemannian-cs-closed-loop}. 
Lemma~\ref{lem:adversary-existence} implies that the adversary condition (section~\ref{subsec:underactuated-proof-prelims}) does not hold for \eqref{eq:f-cd-specific}, \eqref{eq:f-specific-def}, so the adversary theorem (section~\ref{subsec:underactuated-proof-prelims})  implies that $\vo(\cfg)$ is not stabilizable for \eqref{eq:f-cd-specific}, \eqref{eq:f-specific-def}.
In particular, $\vo(\cfg)$ is not asymptotically stable for \eqref{eq:non-riemannian-cs-closed-loop} and hence also not for \eqref{eq:riemannian-cs-closed-loop}.
\end{proof}

\section{Discussion}\label{sec:discussion}
In this paper, we established results determining whether prescribed vector fields on compact smooth manifolds, viewed geometrically as sections of the tangent bundle, can be asymptotically stabilized in a strong sense by general second-order control systems on the base manifold.
We also illustrated the abstract results concretely in Examples~\ref{ex:naim-stabilization}, \ref{ex:non-stabilization} and Figure~\ref{fig:naim-stabilization}.

For a broad class of fully actuated systems, we showed in Theorem~\ref{th:naim-stabilization} that  the section corresponding to any smooth vector field can be made globally exponentially stable, normally attracting, and more, which implies that the closed-loop second-order system asymptotically reproduces the prescribed first-order dynamics in a strong sense.
Moreover, an explicit feedback law to achieve this was provided by extending the classical Koditschek construction \cite[sec.~3.3]{koditschek1987adaptive}  to our geometric setting (see Proposition~\ref{prop:suff-cond-assump:naim-stabilization}).

In contrast, for a broad class of underactuated systems on manifolds of nonzero Euler characteristic, we showed in Theorem~\ref{th:non-stabilizability} that  sections corresponding to ``almost all'' smooth vector fields cannot even be locally asymptotically stabilized.
This includes, in particular, all vector fields having only isolated zeros.
Here the underactuation requirement is that the control system takes values in a $C^0$ regular distribution of positive corank over a neighborhood of the zeros of $\vo$, a common situation for underactuated systems.
We also showed that the Euler characteristic assumption is necessary for Theorem~\ref{th:non-stabilizability} (see Remark~\ref{rem:zero-euler-char-case} and Example~\ref{ex:euler-char-necessary}).

It would be interesting to determine the optimal hypotheses under which the conclusions of Theorems~\ref{th:naim-stabilization} and \ref{th:non-stabilizability} apply.
For the latter, an extension of Coron's necessary condition for stabilizability of points in Euclidean spaces \cite[Thm~2]{coron1990necessary} introduced in \cite[Thm~2]{kvalheim2023obstructions} (see also \cite[Thm~3]{kvalheim2025relationships}) might be useful, since it is stronger than the extension \cite[Thm~3.2]{kvalheim2022necessary} of Brockett's condition \cite[Thm~1.(iii)]{brockett1983asymptotic} used to prove Theorem~\ref{th:non-stabilizability} under fairly general (but not all \cite[Ex.~4]{kvalheim2025relationships}) conditions \cite[Thm~4]{kvalheim2025relationships}.
Additionally, it would be interesting to pursue geometric extensions of other techniques for embedding first-order dynamics in second-order systems distinct from the Koditschek construction \cite{revzen2012dynamical}.

\section*{Acknowledgments}
This material is based upon work supported by the Air Force Office of Scientific Research under award number FA9550-24-1-0299, managed by Dr. Frederick A. Leve.

%\MK{other potentially relevant refs: \cite{welde2023compositional}      \cite{wilson1969smooth,fathi2019smoothing} \cite{gobbino2001topological,jongeneel2024topological} \cite{milnor1965topology,pugh1968generalized, guillemin1974differential}}

\bibliographystyle{amsalpha}
\bibliography{ref}

\end{document}